\documentclass[11pt]{amsart}

\usepackage{amsmath, pb-diagram}
\usepackage{amssymb}
\usepackage{amscd}

\begin{document}

\title{A Generalization of Rickart Modules}

\author{Burcu Ungor}
\address{Burcu Ungor, Department of Mathematics, Ankara University, Turkey}
\email{bungor@science.ankara.edu.tr}

\author{Sait Hal\i c\i oglu}
\address{Sait Hal\i c\i oglu,  Department of Mathematics, Ankara University, Turkey}
\email{halici@ankara.edu.tr}

\author{Abdullah Harmanci}
\address{Abdullah Harmanci, Department of Maths, Hacettepe University, Turkey}
\email{harmanci@hacettepe.edu.tr}

\date{}
\newtheorem{thm}{Theorem}[section]
\newtheorem{lem}[thm]{Lemma}
\newtheorem{prop}[thm]{Proposition}
\newtheorem{cor}[thm]{Corollary}
\newtheorem{exs}[thm]{Examples}
\newtheorem{defn}[thm]{Definition}
\newtheorem{nota}{Notation}
\newtheorem{rem}[thm]{Remark}
\newtheorem{ex}[thm]{Example}

\begin{abstract}
 Let $R$ be an arbitrary ring with
identity and $M$ a right $R$-module with $S=$ End$_R(M)$. In this
paper  we  introduce  $\pi$-Rickart modules as a generalization of
generalized right principally projective rings
 as well as that of Rickart modules. The module $M$ is called
{\it $\pi$-Rickart} if for any $f\in S$, there exist $e^2=e\in S$
and a positive integer $n$ such that $r_M(f^n)=eM$. We prove that
several results of Rickart modules can be extended to
$\pi$-Rickart modules for this general settings, and investigate
relations between a $\pi$-Rickart module and its endomorphism
ring.

 \vspace{2mm}

 \noindent {\bf Key
words}:  $\pi$-Rickart modules, Fitting modules, generalized right
principally projective rings.
\end{abstract}\maketitle

\section{Introduction}
\vskip 0.4 true cm \indent\indent

\indent\indent Throughout this paper $R$ denotes an associative
ring with identity, and modules are unitary right $R$-modules.
 For a module $M$,  $S=$ End$_R(M)$ is the ring of all right $R$-module endomorphisms of
$M$. In this work, for the $(S, R)$-bimodule $M$, $l_S(.)$ and
$r_M(.)$ are the left annihilator of a subset of $M$ in $S$ and
the right annihilator of a subset of $S$ in $M$, respectively.
 A ring is called {\it reduced} if it has no nonzero nilpotent elements. By considering
the right $R$-module $M$ as an $(S, R)$-bimodule the reduced ring
concept was considered for modules in \cite{AHH3}. The module $M$
is called {\it reduced } if for any $f\in S$ and $m\in M$, $fm =
0$ implies $fM\cap Sm = 0$. In \cite{Ka} {\it Baer rings} are
introduced as rings in which the right (left) annihilator of every
nonempty subset is generated by an idempotent. Principally
projective rings were introduced by Hattori \cite{Hat} to study
the torsion theory, that is,  a ring is called  {\it left (right)
principally projective} if every principal left (right) ideal is
projective. The concept of left (right) principally projective
rings (or left (right) Rickart rings) has been comprehensively
studied in the literature. Regarding a generalization of Baer
rings as well as
 principally projective rings, recall that a ring $R$ is called
{\it generalized left (right) principally projective} if for any
$x\in R$, the left (right) annihilator of $x^n$ is generated by an
idempotent for some positive integer n. A number of papers have
been written on generalized principally projective rings (see
\cite{Hir} and  \cite{HKL}).  According to Rizvi and Roman,  an
$R$-module $M$ is called {\it Baer} \cite{TR} if for any
$R$-submodule $N$ of $M$, $l_S(N)=Se$ with $e^2=e\in S$, while the
module $M$  is said to be {\it Rickart} \cite{TR3}  if for any
$f\in S$, $r_M(f)=eM$ for some $e^2=e\in S$. Recently, Rickart
modules are studied extensively by different authors (see
\cite{AHH3} and \cite{LRR}).

In what follows, we denote  by $\Bbb Z$ and  $\Bbb Z_n$ integers
and  the ring of integers modulo $n$, respectively, and $J(R)$
denotes the Jacobson radical of a ring $R$.

\section{$\pi$-Rickart Modules}
In this section, we introduce the concept of $\pi$-Rickart
modules. We supply an example to show that all $\pi$-Rickart
modules need not be Rickart. Although  every direct summand of a
$\pi$-Rickart module is $\pi$-Rickart, we present an example to
show that a direct sum of $\pi$-Rickart modules is not
$\pi$-Rickart. It is shown that the class of some abelian
$\pi$-Rickart modules is closed under direct sums. We begin with
our main definition.

\begin{defn} {\rm Let $M$ be an $R$-module with $S=$ End$_R(M)$. The module $M$ is
called {\it $\pi$-Rickart} if for any $f\in S$, there exist
$e^2=e\in S$ and a positive integer $n$ such that  $r_M(f^n)=eM$.}
\end{defn}
For the sake of brevity, in the sequel, $S$ will stand for the
endomorphism ring of the module $M$ considered.
\begin{rem}\label{son} $R$ is a $\pi$-Rickart $R$-module if and only
if $R$  is a generalized right principally projective ring.
\end{rem}
Every module of finite length, every semisimple, every nonsingular
injective (or extending) and every Baer module is a $\pi$-Rickart
module. Also every quasi-projective strongly co-Hopfian module,
every quasi-injective strongly Hopfian module, every Artinian and
Noetherian module is $\pi$-Rickart (see Corollary \ref{fgric}).
Every finitely generated module over a right Artinian ring is
$\pi$-Rickart (see Proposition \ref{artin}), every free module
which its endomorphism ring is  generalized right principally
projective is $\pi$-Rickart (see Corollary \ref{free}), every
finitely generated projective regular module is $\pi$-Rickart (see
Corollary \ref{fgprojreg}) and every finitely generated projective
module over a commutative $\pi$-regular ring is $\pi$-Rickart (see
Proposition \ref{fgproj}).

One may suspect that every $\pi$-Rickart module is Rickart. But
the following example illustrates that this is not the case.

\begin{ex}\rm{ Consider $M=\Bbb Z \oplus \Bbb{Z}_2$ as a $\Bbb
Z$-module. It can be easily determined that
$S=$ End$_\Bbb Z(M)$ is $\left [\begin{array}{cc}\Bbb Z & 0\\
\Bbb{Z}_2& \Bbb {Z}_2 \end{array}\right]$. For any $f=\left [\begin{array}{cc} a & 0\\
\overline{b}& \overline{c} \end{array}\right]\in S$, consider the
following cases.

$\texttt{Case 1}$. Assume that
$a=0,~\overline{b}=\overline{0},~\overline{c}=\overline{1}$ or
$a=0,~\overline{b}=\overline{c}=\overline{1}$. In both cases $f$
is an idempotent, and so  $r_M(f)=(1-f)M$.

$\texttt{Case 2}$. If $a\neq
0,~\overline{b}=\overline{0},~\overline{c}=\overline{1}$ or $a\neq
0,~\overline{b}=\overline{c}=\overline{1}$, then $r_M(f)=0$.

$\texttt{Case 3}$. If $a\neq
0,~\overline{b}=\overline{c}=\overline{0}$ or $a\neq
0,~\overline{b}=\overline{1},~\overline{c}=\overline{0}$, then
$r_M(f)=0\oplus \Bbb{Z}_2$.

$\texttt{Case 4}$. If $a=0,~\overline{b}=\overline{1}~,
\overline{c}=\overline{0}$, then $f^2=0$. Hence $r_M(f^2)=M$.

 Therefore $M$ is a $\pi$-Rickart module, but it is not Rickart by \cite{LRR}. }\end{ex}

Our next endeavor is to find conditions under which a
$\pi$-Rickart module is Rickart. We show that reduced rings play
an important role in this direction.

\begin{prop} If $M$ is a Rickart module, then it is $\pi$-Rickart. The converse holds if
$S$ is a reduced ring.
\end{prop}

\begin{proof}  The first assertion is clear. For the second, let $f\in S$. Since
$M$ is  $\pi$-Rickart, $r_M(f^n)=eM$ for some positive integer $n$
and $e^2=e\in S$. If $n=1$, then there is nothing to do. Assume
that $n>1$. Since $S$ is a reduced ring, $e$ is central and so
$(fe)^n=0$. It follows that $fe=0$. Hence $eM\leq r_M(f)$. On the
other hand, always $r_M(f)\leq r_M(f^n)=eM$. Therefore $M$ is
Rickart.
\end{proof}

Reduced modules are studied in \cite{AHH3} and it is shown that if
$M$ is a reduced module, then $S$ is a reduced ring. Hence we have
the following.

\begin{cor}\label{red} If $M$ is a reduced module, then it is
Rickart if and only if it is $\pi$-Rickart.
\end{cor}

We obtain the following well known result (see \cite[Lemma
1]{HKL}).

\begin{cor} Let $R$ be a reduced ring.  Then  $R$ is a right Rickart
ring  if and only if $R$ is a generalized right principally
projective ring.
\end{cor}

\begin{lem}\label{idemp} If
$M$ is a $\pi$-Rickart module, then every non-nil left annihilator
in $S$ contains a nonzero idempotent.
\end{lem}
\begin{proof} Let $I=l_S(N)$ be a non-nil left annihilator where $\emptyset \neq
N\subseteq M$ and choose $f\in I$ be a non-nilpotent element. As
$M$ is $\pi$-Rickart, $r_M(f^n)=eM$ for some idempotent $e\in S$
and a positive integer $n$. In addition $e\neq 1$. Due to
$r_M(I)\subseteq r_M(f^n)$, we have $(1-e)r_M(I)=0$. It follows
that $1-e\in l_S(r_M(I))=l_S(r_M(l_S(N)))=l_S(N)=I$. This
completes the proof.
\end{proof}

We now give a relation among $\pi$-Rickart modules, Rickart
modules and Baer modules by using Lemma \ref{idemp}.

\begin{thm}\label{ort-idemp} Let $M$ be a module. If $S$ has no infinite set of nonzero orthogonal
idempotents and $J(S)=0$, then the following are equivalent.
\begin{enumerate}
    \item[{\rm (1)}] $M$ is a  $\pi$-Rickart module.
   \item[{\rm (2)}] $M$ is a Rickart module.
   \item[{\rm (3)}] $M$ is a Baer module.
\end{enumerate}
\end{thm}
\begin{proof} It is enough to show that (1) implies (3). Consider
any left annihilator $I=l_S(N)$ where $\emptyset \neq N\subseteq
M$. If $I$ is nil, then $I\subseteq J(S)$, and so $I=0$. Thus we
may assume that $I$ is not nil. By \cite[Proposition 6.59]{L}, $S$
satisfies DCC on left direct summands, and so among all nonzero
idempotents in $I$, choose $e\in I$ such that $S(1-e)=l_S(eM)$ is
minimal. We claim that $I\cap l_S(eM)=0$. Note that $I\cap
l_S(eM)=l_S(N\cup eM)$. If $I\cap l_S(eM)$ is nil, then there is
nothing to do. Now we assume that $I\cap l_S(eM)$ is not nil. If
$I\cap l_S(eM)\neq 0$, then there exists $0\neq f=f^2\in I\cap
l_S(eM)$ by Lemma \ref{idemp}. Since $fe=0$, $e+(1-e)f\in I$ is an
idempotent, say $g=e+(1-e)f$. Then $ge=e$, and so $g\neq 0$. Also
$fg=f$. This implies that $l_S(gM)\subsetneq l_S(eM)$. This
contradicts to the choice of $e$. Hence $I\cap l_S(eM)=0$. Due to
$\varphi(1-e)\in I\cap l_S(eM)$ for any $\varphi\in I$, we have
$\varphi=\varphi e$. Thus $I\subseteq Se$, and clearly
$Se=I=l_S(N)$. Therefore $M$ is Baer.
\end{proof}

\begin{cor} Let $R$ be a ring. If $R$ has no infinite set of nonzero orthogonal
idempotents and $J(R)=0$, then the following are equivalent.
\begin{enumerate}
    \item[{\rm (1)}] $R$ is a generalized right principally projective ring.
   \item[{\rm (2)}]  $R$ is a right Rickart ring.
   \item[{\rm (3)}] $R$ is a Baer ring.
\end{enumerate}
\end{cor}

\begin{cor} Let $M$ be a module. If $S$ is a semisimple ring, then the following are
equivalent.
\begin{enumerate}
    \item[{\rm (1)}] $M$ is a  $\pi$-Rickart module.
   \item[{\rm (2)}] $M$ is a Rickart module.
   \item[{\rm (3)}] $M$ is a Baer module.
\end{enumerate}
\end{cor}

\begin{proof} Since $S$ is semisimple, we have $J(S)=0$ and $S$ is
left Artinian. Then $S$ has no infinite set of nonzero orthogonal
idempotents by \cite[Proposition 6.59]{L}. Hence Theorem
\ref{ort-idemp} completes the proof.
\end{proof}

\begin{cor} If $M$ is  Noetherian (Artinian) and $J(S)=0$, then
the following are equivalent.
\begin{enumerate}
    \item[{\rm (1)}] $M$ is a  $\pi$-Rickart module.
   \item[{\rm (2)}] $M$ is a Rickart module.
   \item[{\rm (3)}] $M$ is a Baer module.
\end{enumerate}
\end{cor}
\begin{proof} $S$ has no infinite set of nonzero orthogonal
idempotents in case $M$ is  either Noetherian or Artinian. The
rest is clear from Theorem \ref{ort-idemp}.
\end{proof}

Modules which contain  $\pi$-Rickart modules need not be
$\pi$-Rickart, as the following example shows.

\begin{ex}\label{orn2}{\rm Let $R$ denote the ring $\left[\begin{array}{ll}\Bbb
Z&\Bbb Z\\ 0&\Bbb Z
\end{array}\right]$ and $M$ the right $R$-module
$\left [\begin{array}{cc}\Bbb Z & \Bbb Z\\\Bbb Z & \Bbb Z
\end{array}\right]$. Let $f\in S$ be defined by $f\left[\begin{array}{cc}x &
y\\r & s\end{array}\right]= \left[\begin{array}{cc}2x+3r &
2y+3s\\0 & 0\end{array}\right]$, where $\left[\begin{array}{cc}x &
y\\r & s\end{array}\right]\in M$. Then $r_M(f)=\left\{
\left[\begin{array}{rr} 3k & 3z\\-2k & -2z\end{array}\right] : k,
z\in \Bbb Z \right\}$. Since $r_M(f)$ is not a direct summand of
$M$ and  $r_M(f)=r_M(f^n)$ for any integer $n\geq 2$,  $M$ is not
a $\pi$-Rickart module. On the other hand,
consider the submodule $N=\left[\begin{array}{ll}\Bbb Z&\Bbb Z\\
0&0
\end{array}\right]$ of $M$. Then End$_R(N)=\left[\begin{array}{ll}\Bbb Z&0\\ 0&0
\end{array}\right]$. It is easy to show that $N$ is a Rickart module and so
it is $\pi$-Rickart.}
\end{ex}

In \cite[Proposition 2.4]{AHH3}, it is shown that every direct
summand of a Rickart module is Rickart. We now prove that every
direct summand of a $\pi$-Rickart module inherits this  property.

\begin{prop}\label{smnd} Every direct summand of a $\pi$-Rickart module is
\linebreak $\pi$-Rickart.
\end{prop}

\begin{proof} Let $M = N\oplus P$ be an  $R$-module with $S=$
End$_R(M)$ and $S_{N}=$ End$_{R}(N)$. For any $f\in S_{N}$, define
$g = f\oplus 0_{|_{P}}$ and so $g\in S$. By hypothesis, there
exist a positive integer $n$ and $e^2 = e\in S$ such that
$r_M(g^n) = eM$ and $g^n = f^n \oplus 0_{|_{P}}$. Let $M =
eM\oplus Q$. Since $P\subseteq eM$, there exists $L\leq eM$ such
that $eM = P\oplus L$. So we have $M = eM\oplus Q = P\oplus
L\oplus Q$. Let $\pi_{N}: M\rightarrow N$ be the projection of $M$
onto $N$. Then  $\pi_{N}\mid_{Q\oplus L}: Q\oplus L\rightarrow N$
is an isomorphism. Hence $N = \pi_{N}(Q)\oplus \pi_{N}(L)$. We
claim that $r_{N}(f^n) = \pi_{N}(L)$. We get $g^n(L)=0$  since
$g^n(P\oplus L)=0$. But for all $l\in L$,
$l=\pi_{N}(l)+\pi_{P}(l)$. Since $g^n(l) = g^n\pi_{N}(l) +
g^n\pi_{P}(l)$ and $g^n(l) = 0$ and $g^n\pi_{P}(l) = 0$ and
$g^n\pi_{N}(l)= f^n\pi_{N}(l)$, we have $f^n\pi_{N}(L) = 0$ and so
 $\pi_{N}(L)\subseteq r_{N}(f^n)$. For the reverse inclusion, let
 $n\in r_{N}(f^n)$. Assume that $n\notin\pi_{N}(L)$
and we  reach a contradiction. Then $n=n_{1}+n_{2}$ for some
$n_{1}\in \pi_{N}(L)$ and some $0\neq n_{2}\in \pi_{N}(Q)$ and so
there exists a $q\in Q$ such that $\pi_{N}(q) = n_{2}$. Since
$Q\cap r_{M}(g^n)=0$, we have $g^n(q) = (f^n\oplus
0_{|_{P}})(q)\neq 0$. Due to $q = \pi_{N}(q)+ \pi_{P}(q)$ and
$g^n\pi_{P}(q) = ( f^n \oplus 0_{|_{P}})\pi_{P}(q) = 0$, we get
$f^n(q) = g^n(q) = f^n\pi_{N}(q)\neq 0$. This implies $n\notin
r_{N}(f^n)$ which is the required contradiction. Hence
$r_{N}(f^n)\leq \pi_{N}(L)$. Therefore $r_{N}(f^n)=\pi_{N}(L)$.
\end{proof}

\begin{cor}\label{smndcor} Let $R$ be a
generalized right principally projective ring  with any idempotent
$e$ of $R$. Then $eR$ is a $\pi$-Rickart module.
 \end{cor}

\begin{cor} Let $R = R_1\oplus R_2$ be a generalized right principally
projective ring with direct sum of the rings $R_1$ and $R_2$. Then
the rings $R_1$ and $R_2$ are generalized right principally
projective.
\end{cor}

We now characterize generalized right principally projective rings
in terms of  $\pi$-Rickart modules.
\begin{thm} Let $R$ be a ring. Then $R$ is generalized right principally projective if and
only if every cyclic projective $R$-module is  $\pi$-Rickart.
\end{thm}
\begin{proof} The sufficiency is clear. For the necessity, let $M$ be a cyclic projective $R$-module.
Then $M\cong I$ for some direct summand right ideal $I$ of $R$. By
Remark \ref{son}, $R$ is $\pi$-Rickart as an $R$-module and by
Proposition \ref{smnd}, $I$ is $\pi$-Rickart,  and so is $M$.
\end{proof}

\begin{thm} Let $R$ be a ring and consider the following conditions.
\begin{enumerate}
    \item [{\rm (1)}] Every free $R$-module is
    $\pi$-Rickart.
    \item [{\rm (2)}] Every projective $R$-module is
    $\pi$-Rickart.
     \item [{\rm (3)}] Every flat $R$-module is
    $\pi$-Rickart.
\end{enumerate}
Then {\rm (3) $\Rightarrow$ (2) $\Leftrightarrow$ (1)}. Also {\rm
(2) $\Rightarrow$ (3)} holds for finitely presented modules.
\end{thm}
\begin{proof} (3) $\Rightarrow$  (2) $\Rightarrow $ (1) Clear. (1) $\Rightarrow $ (2)
Let $M$ be a projective $R$-module. Then $M$ is a direct summand
of a free $R$-module $F$. By (1), $F$ is $\pi$-Rickart, and so is
$M$ due to Proposition \ref{smnd}.

(2) $\Rightarrow$ (3) is clear from the fact that finitely
presented flat modules are projective.
\end{proof}

\begin{lem}\label{abel} Let $M$ be a  module  and $f\in S$. If $r_M(f^n)=eM$ for some central
idempotent $e \in S$ and  a positive integer $n$, then
 $r_M(f^{n+1})=eM$.
\end{lem}

\begin{proof} It is clear that $r_M(f^n)\leq r_M(f^{n+1})$. For the reverse inclusion,
let $m\in r_M(f^{n+1})$. Then $fm\in r_M(f^n)=eM$, and so
$fm=efm$. Since $e$ is central,
 $f^nm=f^{n-1}fm=f^{n-1}efm=f^{n-1}fem=f^nem=0$. Hence $m\in
r_M(f^n)$ and so $r_M(f^{n+1})\leq r_M(f^n)$.
\end{proof}

The next example reveals that a direct sum of $\pi$-Rickart
modules need not be $\pi$-Rickart.

\begin{ex}\label{orn3}{\rm Let $R$ denote the ring $\left[\begin{array}{ll}\Bbb
Z&\Bbb Z\\ 0&\Bbb Z\end{array}\right]$ and $M$ the right $R$-module$\left [\begin{array}{cc}\Bbb Z & \Bbb Z\\\Bbb Z & \Bbb Z\end{array}\right]$. Consider the submodules $N=\left[\begin{array}{ll}\Bbb Z&\Bbb Z\\
0&0\end{array}\right]$ and $K = \left[\begin{array}{ll}0 &0 \\
\Bbb Z & \Bbb Z\end{array}\right]$ of $M$. It is easy to check
that every nonzero endomorphism of $N$ and $K$ is a monomorphism.
Therefore $N$ and $K$ are $\pi$-Rickart modules but, as was
claimed in Example \ref {orn2}, $M = N\oplus K$ is not
$\pi$-Rickart.}
\end{ex}
A ring $R$ is called {\it abelian} if every idempotent is central,
that is, $ae=ea$ for any $a, e^2=e  \in R$. A module $M$ is called
{\it abelian} \cite{Ro} if $fem = efm$ for any $f\in S$, $e^2 =
e\in S$, $m\in M$. Note that $M$ is an abelian module if and only
if $S$ is an abelian ring.  In \cite[Proposition 7]{HKL}, it is
shown that  the class of abelian generalized right principally
projective rings is closed under direct sums. We extend this
result as follows.

\begin{prop} Let $M_1$ and $M_2$ be $\pi$-Rickart $R$-modules. If
$M_1$ and $M_2$ are abelian  and Hom$_R(M_i, M_j)=0$ for $i\neq
j$, then $M=M_1\oplus M_2$ is a $\pi$-Rickart module.
\end{prop}
\begin{proof} Let $S_i=$ End$_R(M_i)$ for $i=1, 2$ and  $S=$
End$_R(M)$. We may write $S$ as $ \left[\begin{array}{cc}S_1 &
0\\0 & S_2\end{array}\right]$. Let $\left[\begin{array}{cc}f_1 &
0\\0 & f_2\end{array}\right]\in S$ with $f_1\in S_1$ and $f_2\in
S_2$. Then there exist positive integers $n, m$ and $e_1^2=e_1\in
S_1$, $e_2^2=e_2\in S_2$ with
 $r_{M_1}(f_1^n)=e_1M_1$ and $r_{M_2}(f_2^m)=e_2M_2$. Consider the following cases:

 ({\it i}) If  $n=m$, then  obviously  $r_M \left (\left[\begin{array}{cc}f_1 &
0\\0 & f_2\end{array}\right]^n\right )=\left[\begin{array}{cc}e_1
& 0\\0 & e_2\end{array}\right]M$.

({\it ii}) If $n<m$, then by Lemma \ref{abel}, we have
$r_{M_1}(f_1^n)=r_{M_1}(f_1^m)=e_1M_1$. Thus
$\left[\begin{array}{cc}f_1 & 0\\0 & f_2\end{array}\right]^m
\left[\begin{array}{cc}e_1 & 0\\0 & e_2\end{array}\right]=0$, and
so $\left[\begin{array}{cc}e_1 & 0\\0 & e_2\end{array}\right]M\leq
r_M\left( \left[\begin{array}{cc}f_1 & 0\\0 &
f_2\end{array}\right]^m \right)$. Now let $\left[\begin{array}{c}
m_1\\ m_2\end{array}\right]\in r_M\left(
\left[\begin{array}{cc}f_1 & 0\\0 & f_2\end{array}\right]^m
\right)$. Then $m_1\in r_{M_1}(f_1^m)=e_1M_1$ and $m_2\in
r_{M_2}(f_2^m)=e_2M_2$.  Hence $\left[\begin{array}{c} m_1\\
m_2\end{array}\right]= \left[\begin{array}{cc}e_1 & 0\\0 &
e_2\end{array}\right]\left[\begin{array}{c} m_1\\
m_2\end{array}\right]$. Thus $\left[\begin{array}{c} m_1\\
m_2\end{array}\right]\in \left[\begin{array}{cc}e_1 & 0\\0 &
e_2\end{array}\right]M$. Therefore $r_M\left(
\left[\begin{array}{cc}f_1 & 0\\0 & f_2\end{array}\right]^m
\right)\leq \left[\begin{array}{cc}e_1 & 0\\0 &
e_2\end{array}\right]M$.

({\it iii})  If $m<n$, then  the proof is similar to case (ii),
since $M_2$ is abelian.
 \end{proof}

Recall that a  module $M$ is called {\it duo} if every submodule
of $M$ is fully invariant, i.e., for a submodule $N$ of $M$,
$f(N)\leq N$ for each $f\in S$. Our next aim is to find some
conditions under which a fully invariant submodule of a
$\pi$-Rickart module is also $\pi$-Rickart.

\begin{lem}\label{duo1} Let $M$ be a module  and $N$ a fully invariant submodule of $M$. If $M$ is
$\pi$-Rickart  and every endomorphism of $N$ can be extended to an
endomorphism of $M$, then $N$ is $\pi$-Rickart.
\end{lem}
\begin{proof} Let $S=$
End$_R(M)$ and $f\in $ End$_R(N)$. By hypothesis, there exists
$g\in S$ such that  $g|_N=f$ and being $M$ $\pi$-Rickart, there
exist a positive integer $n$ and an idempotent $e$ of $S$ such
that $r_M(g^n)=eM$. Then $r_N(f^n)=N\cap r_M(g^n)$. Since $N$ is
fully invariant, we have $r_N(f^n)=eN$, and so $r_N(f^n)$  is a
direct summand of $N$. Therefore $N$ is $\pi$-Rickart.
\end{proof}

The following result is an immediate consequence of Lemma
\ref{duo1}.

\begin{prop} Let $M$ be a quasi-injective module and $E(M)$ denote the injective
hull of $M$. If $E(M)$ is $\pi$-Rickart, then so is $M$.
\end{prop}

\begin{thm}\label{duo2} Let $M$ be a quasi-injective duo module.  If $M$ is $\pi$-Rickart, then every submodule of $M$ is
$\pi$-Rickart.
\end{thm}
\begin{proof} Let $M$ be a $\pi$-Rickart module and $N$ a submodule of $M$  and  $f\in$ End$_R(N)$. By quasi-injectivity of $M$, $f$ extends to an endomorphism $g$ of $M$.
Then $r_M(g^n) = eM$ for some positive integer $n$ and $e^2 = e\in
S$. Since $N$ is fully invariant under $g$, the proof follows from
Lemma \ref{duo1}.
\end{proof}

Rizvi and Roman \cite{TR} introduced that the module $M$
is ${\mathcal K}$-{\it nonsingular} if for any $ f \in S$,
$r_M(f)$ is essential in $M$ implies $f=0$. They proved that every
Rickart module is ${\mathcal K}$-nonsingular. In order to obtain a
similar result for $\pi$-Rickart modules,  we now give a
generalization of the notion of ${\mathcal K}$-nonsingularity. The
module $M$ is called {\it generalized} ${\mathcal K}$-{\it
nonsingular}, if $r_M(f)$ is essential in $M$ for any $f \in S$,
then $f$ is nilpotent. It is clear that every ${\mathcal
K}$-nonsingular module is generalized ${\mathcal K}$-nonsingular.
The converse holds if the module is rigid. A ring $R$ is called
{\it $\pi$-regular} if for each $a\in R$ there exist a positive
integer $n$ and an element $x$ in $R$ such that $a^n=a^nxa^n$.

\begin{lem}\label{prgk}  Let $M$ be a  module. If $S$ is a $\pi$-regular ring, then $M$ is
 generalized ${\mathcal K}$-nonsingular.
\end{lem}

\begin{proof} Let $f \in S$ with $r_M(f)$  essential in $M$.
By hypothesis, there exist a positive integer $n$ and $g\in S$
such that $f^n=f^ngf^n$. Then  $gf^n$ is an idempotent of $S$ and
so $r_M(f^n)$ is a direct summand of $M$. Since $r_M(f)$ is
essential in $M$, $r_M(f^n)$ is also essential  in $M$. Hence
$r_M(gf^n)=M$ and so  $gf^n=0$. Therefore $f^ngf^n=f^n=0$.
\end{proof}

\begin{prop}\label{nonsin} Every $\pi$-Rickart module is generalized
${\mathcal K}$-nonsingular.
\end{prop}

\begin{proof} Let $M$ be a $\pi$-Rickart module  and $f \in S$ with $r_M(f)$ essential in $M$.
Then  $r_M(f^n)=eM$ for some $e^2=e\in S$ and a positive integer
$n$. Hence $r_M(f^n)$ is essential in $M$. Thus $r_M(f^n)=M$ and
so $f^n=0$.
\end{proof}

\begin{cor} If $R$ is a generalized right principally projective ring,
then $R$ is  generalized ${\mathcal K}$-nonsingular as an
$R$-module.
\end{cor}

Our next purpose is to find out the conditions when a
$\pi$-Rickart module $M$ is torsion-free as an $S$-module. So we
consider the set $T(_SM)=\{m\in M\mid fm=0$ for some nonzero $f\in
S\}$ of all torsion elements of a module $M$ with respect to $S$.
The subset $T(_SM)$ of $M$ need not be a submodule of the modules
$_SM$ and $M_R$ in general. If $S$ is a commutative domain, then
$T(_SM)$ is an $(S, R)$-submodule of $M$.

\begin{prop} Let $M$ be a module with a commutative domain $S$.
If $M$ is  $\pi$-Rickart, then $T(_SM)=0$ and every nonzero
element of $S$ is a  monomorphism.
\end{prop}

\begin{proof} Let $0\neq f\in S$. Then there exist a positive integer $n$ and  $e^2=e\in
S$ such that $r_M(f^n) = eM$. Hence $f^ne= 0$. Since $S$ is a
domain, we have $e = 0$ and so $r_M(f^n) = 0$. This implies that
Ker$f=0$. Thus $f$ is a monomorphism. On the other hand,  if $m\in
T(_SM)$ there exists $0\neq f\in S$ such that $fm=0$. Being $f$  a
monomorphism, we have $m=0$, and so $T(_SM)=0$.
\end{proof}

We close this section with the relations among strongly Hopfian
modules, Fitting modules and $\pi$-Rickart modules. Recall that a
module $M$ is called {\it Hopfian} if every surjective
endomorphism of M is an automorphism, while  $M$ is called {\it
strongly Hopfian} \cite{HKS}  if for any endomorphism $f$ of $M$
the ascending chain Ker$f \subseteq$ Ker$f^2 \subseteq \cdots
\subseteq $  Ker$f^n \subseteq \cdots $ ~~ stabilizes.  We now
give a relation between abelian and strongly Hopfian modules by
using $\pi$-Rickart modules.

\begin{cor}\label{sthop} Every abelian $\pi$-Rickart module is strongly
Hopfian.
\end{cor}

\begin{proof} It follows from Lemma \ref{abel} and
\cite[Proposition 2.5]{HKS}.
\end{proof}

A module  $M$ is said to be {\it a Fitting module} \cite{HKS} if
for any $f \in S$, there exists an integer $n \geq 1$ such that
$M=$ Ker$f^n\oplus $ Im$f^n$. A ring $R$ is called {\it strongly
$\pi$-regular} if for every element $a$ of $R$ there exist a
positive  integer $n$ (depending on $a$) and an element $x$ of $R$
such that $a^n = a^{n+1} x$, equivalently, an element $y$ of $R$
such that $a^n = ya^{n+1}$.  Due to Armendariz, Fisher and Snider
\cite{AFS}, the module $M$ is a Fitting module if and only if $S$
is a strongly $\pi$-regular ring. In this direction we have the
following result.

\begin{cor}\label{fgric} Every Fitting module is a $\pi$-Rickart
module.
\end{cor}

\begin{cor} Let $R$ be a ring and let $n$ be a positive integer. If the
matrix ring $M_n(R)$ is strongly $\pi$-regular, then $R^n$ is a
$\pi$-Rickart $R$-module.
\end{cor}
\begin{proof} Let $M_n(R)$ be a strongly $\pi$-regular ring. Then
by \cite[Corollay 3.6]{HKS}, $R^n$ is a Fitting $R$-module and so
it is $\pi$-Rickart.
\end{proof}

The following provides another source of examples for
$\pi$-Rickart modules.
\begin{prop}\label{artin}
Every finitely generated module over a right Artinian ring is
$\pi$-Rickart.
\end{prop}
\begin{proof} Let $R$ be a right Artinian ring and $M$ a finitely
generated $R$-module. Then $M$ is an Artinian and Noetherian
module. Hence $M$ is a Fitting module. Thus Corollary \ref{fgric}
completes the proof.
\end{proof}

\section{The Endomorphism Ring of a $\pi$-Rickart Module}

In this section we study some relations between $\pi$-Rickart
modules and their endomorphism rings. We prove that endomorphism
ring of a $\pi$-Rickart module is always  generalized right
principally projective, the converse holds either the module is
flat over its endomorphism ring or it is $1$-epiretractable. Also
modules whose endomorphism rings are $\pi$-regular are
characterized.

\begin{lem}\label{grpp1}
If $M$ is a $\pi$-Rickart module, then $S$ is a generalized right
principally projective ring.
\end{lem}

\begin{proof} If $f\in S$, then $r_M(f^n)=eM$ for some $e^2=e\in S$ and positive integer
$n$. If  $g\in r_S(f^n)$, then $gM \leq r_M(f^n)=eM$. This implies
that $g=eg\in eS$, and so $r_S(f^n)\leq eS$. Let $h\in S$. Due to
$f^nehM\leq f^neM=0$, we have $f^neh=0$. Hence $eS\leq r_S(f^n)$.
Therefore $r_S(f^n)=eS$.
\end{proof}

\begin{cor} Let $M$ be a module and $f\in
S$. If $r_M(f^n)$ is a direct summand of $M$ for some positive
integer $n$, then $f^nS$ is a projective right $S$-module.
\end{cor}

The next result (see \cite[Proposition 9]{HKL}) is a consequence
of Theorem \ref{smnd} and Lemma \ref{grpp1}.

\begin{cor} If $R$ is a generalized right principally projective ring,
then  so is $eRe$ for any $e^2=e\in R$.
\end{cor}

A module $M$ is called {\it $n$-epiretractable} \cite{GV}
 if every $n$-generated submodule of $M$ is a homomorphic image of $M$.
We now show that $1$-epiretractable modules allow us to get the
converse of Lemma \ref{grpp1}.

\begin{prop}\label{grpp} Let $M$ be a  $1$-epiretractable module.
Then $M$ is  $\pi$-Rickart if and only if $S$ is a generalized
right principally projective ring.
\end{prop}

\begin{proof} The necessity holds from Lemma \ref{grpp1}. For the
sufficiency, let $f\in S$. Since $S$ is  generalized right
principally projective, there exist a positive integer $n$ and
$e^2=e\in S$ such that $r_S(f^n)=eS$. Then $f^ne=0$, and so
$eM\leq r_M(f^n)$. In order to show the reverse inclusion, let
$0\neq m\in r_M(f^n)$. Being $M$ $1$-epiretractable, there exists
$0\neq g\in S$ with $gM=mR$, and so $m=gm_1$ for some $m_1\in M$.
On the other hand,  $f^ngM=f^nmR=0$, and so $f^ng=0$. Thus $g\in
r_S(f^n)=eS$. It follows that $g=eg$. Hence we have
$m=gm_1=egm_1=em\in eM$. Therefore $r_M(f^n)=eM$.
\end{proof}

\begin{cor}\label{free} A free module is $\pi$-Rickart if and only if its endomorphism
ring is generalized right principally projective.
\end{cor}

A module $M$ is called {\it regular} (in the sense of Zelmanowitz
\cite{Ze}) if for any $m\in M$ there exists a right
$R$-homomorphism $M\stackrel{\phi}\rightarrow R$ such that
\linebreak $m = m\phi(m)$. Every cyclic submodule of a regular
module is a direct summand, and so it is $1$-epiretractable. Then
we have the following result.

\begin{cor}\label{regular} Let $M$ be a regular $R$-module. Then
 $S$ is generalized right principally projective
if and only if $M$ is  $\pi$-Rickart.
\end{cor}

\begin{cor}\label{fgprojreg} Every finitely generated projective regular module is
 $\pi$-Rickart.
\end{cor}
\begin{proof} Let $M$ be a finitely generated projective regular
module. By \cite[Theorem 3.6]{Wa}, the endomorphism ring of $M$ is
a generalized right principally projective ring. Hence, by
Corollary \ref{regular}, $M$ is $\pi$-Rickart.
\end{proof}

Let $\mathcal{U}$ be a nonempty set of $R$-modules. Recall that
for an $R$-module $L$, the submodule 
Tr$ (\mathcal{U},L) = \sum \{ Imh | h \in Hom(U,L), U \in
\mathcal{U}\}$ 
is called the {\it trace of $\mathcal{U}$ in $L$.} If
$\mathcal{U}$ consists of a single module $U$ we simply write
Tr$(U,L)$.

The following result shows that the converse of Lemma \ref{grpp1}
is also true for flat modules over their endomorphism rings. On
the other hand, Theorem  \ref{wis} generalizes the result
\cite[39.10]{Wi1}.

\begin{thm}\label{wis} Let $M$ be an $R$-module and $f\in S$.
Then we have the following.
\begin{enumerate}
\item[(1)] If $f^nS$ is a projective right $S$-module for some positive integer  $n$, then Tr$(M, r_M(f^n))$ is a direct summand of $M$.
\item[(2)] If $M$ is a flat left  $S$-module and $S$ is a generalized right  principally projective ring, then
$M$ is  $\pi$-Rickart as an $R$-module.
\end{enumerate}
\end{thm}
\begin{proof} (1) Assume that $f^nS$ is a projective right $S$-module for some
\linebreak positive integer  $n$. Then there exists $e^2 = e\in S$
with $r_S(f^n) = eS$. We show Tr$(M,r_M(f^n)) = eM$. Since $f^neM
= 0$, $eM\leq $ Tr$(M,r_M(f^n))$. Let \linebreak $g \in $
Hom$(M,r_M(f^n))$. Hence $gM \leq r_M(f^n)$ or $f^ngM = 0$ or
$f^ng = 0$. Thus  $g \in  r_S(f^n) = eS$ and so  $eg = g$. It
follows that $gM\leq egM\leq eM$ or
 Hom$(M,r_M(f^n))M\leq eM$.

(2) Assume that $M$ is a flat left $S$-module and $S$ is a
generalized right principally projective ring. If $f\in S$, then
$f^nS$ is a projective right $S$-module,  since $r_S(f^n) = eS$
for some positive integer $n$ and $e^2 = e\in S$. As in the proof
of (1), we have Tr$(M,r_M(f^n)) = eM$. Since $M$ is a flat left
$S$-module and $f^n\in S$,  $r_M(f^n)$ is $M$-generated by
\cite[15.9]{Wi1}. Again by \cite[13.5(2)]{Wi1}, Tr$(M,r_M(f^n)) =
r_M(f^n)$. Thus $r_M(f^n) = eM$.
\end{proof}

Recall that a ring $R$ is said to be \emph{von Neumann regular} if
for any $a \in R$ there exists $b \in R$ with $a = aba$. For a
module $M$, it is shown that if S is a von Neumann regular ring,
then $M$ is a Rickart module (see \cite[Theorem 3.17]{LRR}). We
obtain a similar result for $\pi$-Rickart modules.

\begin{lem}\label{pireg}  Let $M$ be a  module. If $S$ is a $\pi$-regular ring, then $M$ is a
$\pi$-Rickart module.
\end{lem}

\begin{proof}  Let $f\in S$. Since $S$ is $\pi$-regular,
there exist a positive integer $n$ and an element $g$ in $S$ such
that $f^n=f^ngf^n$. Then  $gf^n$ is an idempotent of $S$. Now we
show that $r_M(f^n)=(1-gf^n)M$. For $m\in M$,  we have
$f^n(1-gf^n)m=(f^n-f^ngf^n)m=(f^n-f^n)m=0$. Hence $(1-gf^n)M\leq
r_M(f^n)$. For the other side, if $m\in r_M(f^n)$, then $gf^nm=0$.
This implies that $m=(1-gf^n)m\in (1-gf^n)M$. Therefore
$r_M(f^n)=(1-gf^n)M$.
\end{proof}

Now we  recall some known facts that will be needed about
$\pi$-regular rings.
\begin{lem}\label{dorsey} Let $R$ be a ring. Then
\begin{enumerate}
    \item[{\rm (1)}] If $R$ is $\pi$-regular, then $eRe$ is also
    $\pi$-regular  for any $e^2=e\in R$.
    \item[{\rm (2)}] If $M_n(R)$ is $\pi$-regular for any positive integer
    $n$, then so is $R$.
    \item[{\rm (3)}] If $R$ is a commutative ring, then $R$ is $\pi$-regular
    if and only if $M_n(R)$ is $\pi$-regular for any positive integer
    $n$.
\end{enumerate}
\end{lem}
\begin{proof} (1) Let $R$ be a $\pi$-regular ring, $e^2=e\in R$ and  $a\in
eRe$. Then $a^n=a^nra^n$ for some positive integer $n$ and $r\in
R$. Since $a^n=a^ne=ea^n$, we have $a^n=a^n(ere)a^n$. Therefore
$eRe$ is $\pi$-regular.

(2) is clear from (1).

(3) Let $R$ be a commutative $\pi$-regular ring. By
\cite[Ex.4.15]{L1}, every prime ideal of $R$ is maximal, and so
every finitely generated $R$-module is co-Hopfian from \cite{V}.
Then for any positive integer $n$, $M_n(R)$ is $\pi$-regular by
\cite[Theorem 1.1]{AFS}. The rest is known from (2).
\end{proof}

\begin{prop}\label{fgproj}
Let $R$ be a commutative $\pi$-regular ring. Then every finitely
generated projective $R$-module is $\pi$-Rickart.
\end{prop}
\begin{proof} Let $M$ be a finitely generated  projective $R$-module. So the endomorphism ring of
$M$ is $eM_n(R)e$ with some positive integer $n$ and an idempotent
$e$ in $M_n(R)$. Since $R$ is commutative  $\pi$-regular, $M_n(R)$
is also  $\pi$-regular, and so is $eM_n(R)e$ by Lemma
\ref{dorsey}. Hence $M$ is $\pi$-Rickart  by Lemma \ref{pireg}.
\end{proof}

The converse of Lemma \ref{pireg} may not be true in general, as
the following example shows.

\begin{ex}{\rm Consider $\Bbb Z$ as a $\Bbb Z$-module. Then it can be easily shown that
$\Bbb Z$ is a $\pi$-Rickart module, but its endomorphism ring is
not  $\pi$-regular.}
\end{ex}

 A module $M$ has {\it $C_2$ condition}
if any submodule $N$ of $M$ which is isomorphic to a direct
summand of $M$ is a direct summand. In \cite[Theorem 3.17]{LRR},
it is proven that the module $M$  is Rickart with $C_2$ condition
if and only if  $S$ is von Neumann regular. The $C_2$ condition
allows us to show  the converse of Lemma \ref{pireg}.

\begin{thm}\label{C2}  Let $M$ be a  module with $C_2$ condition. Then $M$ is $\pi$-Rickart if and
only if $S$ is a $\pi$-regular ring.
\end{thm}
\begin{proof}  The sufficiency holds from Lemma \ref{pireg}. For
the necessity, let $0\neq f\in S$. Since $M$ is $\pi$-Rickart,
Ker$f^n$ is a direct summand of $M$ for some positive integer $n$.
Let $M=$ Ker$f^n\oplus N$ for some $N\leq M$. It is clear that
$f^n|_N$ is a monomorphism. By the $C_2$ condition, $f^nN$ is a
direct summand of $M$. On the other hand, there exists $0\neq g\in
S$ such that $gf^n|_N=1_N$. Hence
$(f^n-f^ngf^n)M=(f^n-f^ngf^n)$(Ker$f^n\oplus N)=(f^n-f^ngf^n)N=0$.
Thus $f^n=f^ngf^n$, and so $S$ is a $\pi$-regular ring.
\end{proof}

The following is a  consequence of Proposition \ref{fgproj} and
Theorem \ref{C2}.
\begin{cor} Let $R$ be a commutative ring and satisfy $C_2$
condition. Then the following are equivalent.\begin{enumerate}
    \item [{\rm (1)}] $R$ is a $\pi$-regular ring.
    \item [{\rm (2)}] Every finitely generated projective
    $R$-module is $\pi$-Rickart.
\end{enumerate}
\end{cor}

As every quasi-injective module has $C_2$ condition, we have the
following.

\begin{cor}\label{prgd} Let $M$ be a quasi-injective module. Then M is $\pi$-Rickart if and only if $S$ is a
$\pi$-regular ring.
\end{cor}

\begin{cor} Every right self-injective ring is generalized right principally
projective if and only if it is $\pi$-regular.
\end{cor}

\begin{thm}  Let $R$ be a right self-injective ring. Then the
following are equivalent.
\begin{enumerate}
    \item[{\rm (1)}] $M_n(R)$ is $\pi$-regular for every positive
integer $n$.
   \item[{\rm (2)}] Every finitely generated projective $R$-module is
   $\pi$-Rickart.
\end{enumerate}
\end{thm}
\begin{proof} (1) $\Rightarrow$ (2) Let $M$ be a finitely generated
projective $R$-module. Then $M\cong eR^n$ for some positive
integer $n$ and $e^2=e\in M_n(R)$. Hence $S$ is isomorphic to
$eM_n(R)e$. By (1), $S$ is $\pi$-regular. Thus $M$ is
$\pi$-Rickart due to Lemma \ref{pireg}.

(2) $\Rightarrow$ (1) $M_n(R)$ can be viewed as the endomorphism
ring of a projective $R$-module $R^n$ for any positive integer
$n$. By (2), $R^n$ is $\pi$-Rickart, and by hypothesis, it is
quasi-injective. Then $M_n(R)$ is $\pi$-regular by Corollary
\ref{prgd}.
\end{proof}

The proof of Lemma \ref{maybe} may be in the context. We include
it as an easy reference.

\begin{lem}\label{maybe} Let $M$ be a module.
Then $S$ is a $\pi$-regular ring if and only if for each $f\in S$,
there exists a positive integer $n$ such that Ker$f^n$ and Im$f^n$
are direct summands of $M$.
\end{lem}

\begin{thm} Let $M$ be a $\pi$-Rickart module. Then the right singular ideal $Z_r(S)$ of $S$ is nil
and $Z_r(S)\subseteq J(S)$.
\end{thm}
\begin{proof} Let $f\in Z_r(S)$. Since $M$ is $\pi$-Rickart,
$r_M(f^n)=eM$ for some positive integer $n$ and $e=e^2\in S$. By
Lemma \ref{grpp1}, $r_S(f^n)=eS$. Since $r_S(f^n)$ is essential in
$S$ as a right ideal,  $r_S(f^n)=S$. This implies that $f^n=0$,
and so $Z_r(S)$ is nil. On the other hand, for any $g\in S$ and
$f\in Z_r(S)$,  according to previous discussion, $(fg)^n=0$ for
some positive integer $n$. Hence $1-fg$ is invertible. Thus $f\in
J(S)$. Therefore $Z_r(S)\subseteq J(S)$.
\end{proof}

\begin{prop}\label{local} The following are equivalent for a module $M$.
\begin{enumerate}
\item[{\rm (1)}] Each element of $S$ is either a monomorphism or nilpotent.
\item[{\rm (2)}] $M$ is an indecomposable $\pi$-Rickart
module.
\end{enumerate}
\end{prop}
\begin{proof} (1) $\Rightarrow$ (2) Let $e=e^2\in S$. If $e$ is
nilpotent, then $e=0$. If $e$ is a monomorphism, then $e(m-em)=0$
implies $em=m$ for any $m\in M$. Hence $e=1$, and so $M$ is
indecomposable. Also for any $f\in S$, $r_M(f)=0$ or $r_M(f^n)=M$
for some positive integer $n$. Therefore $M$ is $\pi$-Rickart.

(2) $\Rightarrow$ (1)  Let $f\in S$. Then $r_M(f^n)$ is a direct
summand of $M$ for some positive integer $n$. As $M$ is
indecomposable, we see that $r_M(f^n)=0$ or $r_M(f^n)=M$. This
implies that $f$ is a monomorphism or nilpotent.
\end{proof}

\begin{thm}\label{pimorphic} Consider the following conditions for a module $M$.
\begin{enumerate}
    \item[{\rm (1)}] $S$ is a local ring with nil Jacobson radical.
   \item[{\rm (2)}] $M$ is an indecomposable $\pi$-Rickart
   module.
\end{enumerate}
Then {\rm (1)} $\Rightarrow$ {\rm (2)}. If $M$ is a morphic
module, then {\rm (2)} $\Rightarrow$ {\rm (1)}.
\end{thm}
\begin{proof} (1) $\Rightarrow$ (2) Clearly, each element of $S$ is either a monomorphism
or nilpotent. Then $M$ is  indecomposable $\pi$-Rickart due to
Proposition \ref{local}.

(2) $\Rightarrow$ (1)  Let $f\in S$. Then $r_M(f^n)=eM$ for some
positive integer $n$ and an idempotent $e$ in $S$. If $e=0$, then
$f$ is a monomorphism. Since $M$ is morphic, $f$ is  invertible by
\cite[Corollary 2]{NC}. If $e=1$, then $f^n=0$. Hence $1-f$ is
invertible. This implies that $S$ is a local ring. Now let $0\neq
f\in J(S)$. Since $f$ is not invertible, there exists a positive
integer $n$ such that $r_M(f^n)=M$. Therefore $J(S)$ is nil.
\end{proof}

The next result can be obtained from Theorem \ref{pimorphic} and
\cite[Lemma 2.11]{HC}.
\begin{cor} Let $M$ be an indecomposable $\pi$-Rickart module. If
$M$ is morphic, then $S$ is a left and right $\pi$-morphic ring.
\end{cor}

In \cite{LRR1}, a module $M$ is called {\it dual Rickart } if for
any $f\in S$, Im$f=eM$ for some $e^2=e\in S$.  In a subsequent
paper the present authors continue studying some generalizations
of dual Rickart modules. The  module $M$ is called {\it  dual
$\pi$-Rickart} if for any $f\in S$, there exist $e^2=e\in S$ and a
positive integer $n$ such that Im$f^n=eM$. We end this paper to
demonstrate the relations between $\pi$-Rickart and dual
$\pi$-Rickart modules.

\begin{prop} Let $M$ be a module with C$_2$ condition.
If $M$ is a $\pi$-Rickart module, then it is dual $\pi$-Rickart.
\end{prop}

\begin{proof} It follows from Theorem \ref{C2} and Lemma
\ref{maybe}. \end{proof}

Recall that a module $M$ is said to have {\it $D_2$ condition} if
any submodule $N$ of $M$ with $M/N$ isomorphic to a direct summand
of $M$, then $N$ is a direct summand of $M$.

\begin{prop} Let $M$ be a module with D$_2$ condition.
 If $M$ is a dual $\pi$-Rickart, then it is $\pi$-Rickart.
\end{prop}
\begin{proof} Since $M/r_M(f^n)\cong$ Im$f^n$ for any $f\in S$, D$_2$ condition
completes the proof.
\end{proof}

\begin{prop} Let $M$ be a projective morphic module. Then $M$ is  $\pi$-Rickart if and only if it is dual
$\pi$-Rickart.
\end{prop}
\begin{proof} Clear.
\end{proof}

\end{document}